\documentclass[11pt]{article}
\usepackage{amsmath,amsthm,amsfonts,amssymb,bm,wasysym}
\usepackage{subfigure}
\usepackage{epsfig}
\usepackage[usenames]{color}
\usepackage{verbatim}
\usepackage{hyperref}

\parskip \medskipamount
\numberwithin{equation}{section}
\newtheorem{theorem}{Theorem}[section]
\newtheorem{lemma}[theorem]{Lemma}
\newtheorem{proposition}[theorem]{Proposition}
\newtheorem{corollary}[theorem]{Corollary}
\theoremstyle{definition}
\newtheorem{definition}[theorem]{Definition}
\newtheorem{remark}[theorem]{Remark}

\numberwithin{equation}{section}

\def\Z{\mathbb{Z}}

\renewcommand{\phi}{\varphi}
\renewcommand{\epsilon}{\varepsilon}

\allowdisplaybreaks

\newcommand{\1}{{\text{\Large $\mathfrak 1$}}}

\renewcommand{\emptyset}{\varnothing}

\newcommand{\tmix}{t_{\mathrm{mix}}}

\newcommand{\thit}{t_{\mathrm{hit}}}
\newcommand{\tmax}{t_{\mathrm{H}}}

\newcommand{\tces}{t_{\mathrm{Ces}}}

\newcommand{\tin}{t_{\mathrm{I}}}
\newcommand{\tunif}{t_{\mathrm{unif}}}

\newcommand{\E}[1]{\mathbb{E}\!\left[#1\right]}
\newcommand{\estart}[2]{\mathbb{E}_{#2}\!\left[#1\right]}
\newcommand{\prstart}[2]{\mathbb{P}_{#2}\!\left(#1\right)}
\newcommand{\prcond}[3]{\mathbb{P}_{#3}\!\left(#1\;\middle\vert\;#2\right)}

\newcommand{\escond}[3]{\mathbb{E}_{#3}\!\left[#1\;\middle\vert\;#2\right]}

\newcommand{\norm}[1]{\left\| #1 \right\|}

\newcommand{\tn}{|\kern-.1em|\kern-0.1em|}

\newcommand\be{\begin{equation}}
\newcommand\ee{\end{equation}}

\begin{document}
\title{\bf Intersection and mixing times for reversible chains}

\author{
Yuval Peres\thanks{Microsoft Research, Redmond, Washington, USA; peres@microsoft.com} \and Thomas Sauerwald\thanks{University of Cambridge, Cambridge, UK; thomas.sauerwald@cl.cam.ac.uk} \and Perla Sousi\thanks{University of Cambridge, Cambridge, UK;   p.sousi@statslab.cam.ac.uk} \and Alexandre Stauffer \thanks{University of Bath, Bath, UK; a.stauffer@bath.ac.uk. Supported in
part by a Marie Curie Career Integration Grant PCIG13-GA-2013-618588 DSRELIS}
}
\date{}
\maketitle
\begin{abstract}
Suppose $X$ and $Y$ are two independent irreducible Markov chains on $n$ states. We consider the intersection time, which is the first time their trajectories intersect. We show for reversible and lazy chains that the total variation mixing time is always upper bounded by the expected intersection time taken over the worst starting states. For random walks on trees we show the two quantities are equivalent. We obtain an expression for the expected intersection time in terms of the eigenvalues for reversible and transitive chains. For such chains we also show that it is up to constants the geometric mean of $n$ and $\E{I}$, where $I$ is the number of intersections up to the uniform mixing time. Finally for random walks on regular graphs we obtain sharp inequalities that relate the expected intersection time to maximum hitting time and mixing time.
\newline
\newline
\emph{Keywords and phrases.} Intersection time, random walk, mixing time, martingale, Doob's maximal inequality.
\newline
MSC 2010 \emph{subject classifications.} Primary 60J10.
\end{abstract}

\section{Introduction}\label{sec:intro}

Intersections of Markov chains have been intensively studied, partly due to their connection with loop-erased walks and spanning trees. The 1991 book of Lawler~\cite{Lawlerinter}  focuses on intersections of random walks on lattices. In  1989, 
Fitzsimmons and Salisbury~\cite{FitzSal} developed techniques for analysing intersections of Brownian motions and L\'evy processes. In 1996, Salisbury~\cite{Salisbury} adapted those techniques in order to bound intersection probabilities for discrete time Markov chains.
 In 2003, Lyons, Pemantle and Peres~\cite{LPSinters} used Salisbury's result to extend certain intersection probability estimates from lattices to general Markov chains.

In this paper we focus on finite Markov chains and study the {\em intersection time},  defined as follows. Let $P$ denote the transition matrix of an irreducible Markov chain on a finite state space, with stationary distribution~$\pi$. Let~$X$ and~$Y$ be two independent Markov chains with transition matrix~$P$.
Define
\[
\tau_I = \inf\{t\geq 0: \{X_0,\ldots, X_t\}\cap \{Y_0,\ldots, Y_t\} \neq \emptyset\},
\]
i.e.\ $\tau_I$ is the first time the trajectories of~$X$ and~$Y$ intersect. The key quantity will be the  expectation of the random time defined above, maximized over starting states:
\[
\tin = \max_{x,y} \estart{\tau_I}{x,y}.
\]
This quantity was considered in~\cite{interndraft}, where it was estimated in many  examples, in particular random walks on   tori $\Z_\ell^d$ for $d\geq 1$.

We denote by $\tmix=\tmix(1/4)$ the total variation mixing time and by $\thit = \max_{x,y}\estart{\tau_y}{x}$ the maximum hitting time, where for all $y$
\[
\tau_y=\inf\{t\geq 0: X_t = y\}.
\]

In order to avoid periodicity and near-periodicity issues we consider the lazy version of a Markov chain, i.e.\ the chain with transition matrix $P_L=(P+I)/2$. From now on, unless otherwise stated, all chains will be assumed lazy.

For functions $f,g$ we will write $f(n) \lesssim g(n)$ if there exists a constant $c > 0$ such that $f(n) \leq c g(n)$ for all $n$.  We write $f(n) \gtrsim g(n)$ if $g(n) \lesssim f(n)$.  Finally, we write $f(n) \asymp g(n)$ if both $f(n) \lesssim g(n)$ and $f(n) \gtrsim g(n)$.

We define $$\tmax = \max_{x,A: \pi(A)\geq 1/8} \estart{\tau_A}{x},$$ where $\tau_A$ stands for the first hitting time of the set $A$.

Our first result shows that $\tin$ is an upper bound on $\tmax$ for all chains. Recall that all our chains are irreducible. 

\begin{theorem}\label{thm:upperboundmix}
For all lazy Markov chains on a finite state space we have
\[
\tmax \lesssim t_I.
\]
\end{theorem}

Using the equivalence between mixing times and $\tmax$ for reversible chains proved independently by~\cite{mixhitroberto} and~\cite{mixhit} we obtain the following corollary.

\begin{corollary}\label{cor:tmix}
For all reversible and lazy Markov chains on a finite state space we have
\[
\tmix \lesssim t_I.
\]
For weighted random walks on finite trees, we have
\[
\tmix \asymp \tin.
\]
\end{corollary}

We prove Theorem~\ref{thm:upperboundmix} and Corollary~\ref{cor:tmix}
in Section~\ref{sec:reversible}, where we also state the equivalence between mixing and hitting times.

\begin{remark}\rm{
We recall the definition of the Cesaro mixing time
\[
\tces=\min\left\{t: \, \max_x\norm{\frac{1}{t}\sum_{i=0}^{t-1}p_s(x,\cdot)  -\pi  }_{{\rm TV}}\leq \frac{1}{4} \right\}.
\]
Since $\tces\asymp \tmax$ for all lazy and irreducible chains without assuming reversibility (see for instance~\cite[Theorem~6.1 and~Proposition~7.1]{mixhit}), it follows from Theorem~\ref{thm:upperboundmix} that $\tces \lesssim \tin$.
}
\end{remark}

\begin{remark}\rm{
We note that $\tin\leq 2\thit$, since we can fix a state and wait until both chains hit it.  So Theorem~\ref{thm:upperboundmix} demonstrates that the intersection time can be sandwiched between the mixing time and the maximum hitting time of the chain. Hence this double inequality can be viewed as a refinement of the basic inequality stating that the mixing time is upper bounded by the maximum hitting time, which is rather loose for many chains.}
\end{remark}

We denote by $\tunif$ the uniform mixing time, i.e.
\[
\tunif = \inf\left\{ t\geq 0: \max_{x,y}\left|\frac{p_t(x,y)}{\pi(y)} -1 \right|\leq \frac{1}{4} \right\},
\]
where $p_t(x,y)$ stands for the transition probability from~$x$ to~$y$ in~$t$ steps for a lazy chain.
Benjamini and Morris~\cite{BenjaminiMorris} related $\tunif$ to intersection properties of multiple random walks.

A chain is called transitive if for any two points $x, y$ in the state space $E$, there is a bijection $\phi:E\to E$ such that $\phi(x)=y$ and $p(z,w) = p(\phi(z),\phi(w))$ for all $z,w$.

For transitive reversible chains, we obtain an expression for the intersection time as stated in the following theorem. We prove it in Section~\ref{sec:transitive}.

\begin{theorem}\label{thm:transitive}
Let $X$ be a transitive, reversible and lazy chain on $n$ states and $Q=\sum_{j=2}^{n} (1-\lambda_j)^{-2}$, where $(\lambda_j)_j$ are the non-unit eigenvalues of the chain in decreasing order. Then we have
\[
\tin \asymp \sqrt{Q}\quad \text{and} \quad Q\asymp n\sum_{i,j=0}^{\tunif} p_{i+j}(x,x)
\]
for any state $x$.
\end{theorem}

\begin{remark}\rm{
Let $X$ and $Y$ be independent transitive, reversible and lazy chains starting from $x$.
We note that if $I=\sum_{i=0}^{\tunif} \sum_{j=0}^{\tunif} \1(X_i=Y_j)$, then $\E{I} =\sum_{i,j=0}^{\tunif}p_{i+j}(x,x)$. So Theorem~\ref{thm:transitive} can be restated by saying
\[
\tin\asymp \sqrt{n\cdot \E{I}}.
\]
}
\end{remark}

\begin{remark}\rm{
For a lazy simple random walk on $\Z_\ell^d$, the local central limit theorem
implies that  $p_t(x,x) \asymp t^{-d/2} $ for each fixed $d$ when  $t \le\tunif \asymp \ell^2$. Thus  the above theorem gives  the intersection time in $\Z_\ell^d$, for any $d \ge 1$.  In particular,
$t_I \asymp \ell^2$ for $d=1,2,3$, while $t_I\asymp \sqrt{n} \log n$ for $d=4$ and $t_I\asymp \sqrt{n}$ for $d \ge 5$, where $n=\ell^d$. These estimates were  derived  in~\cite{interndraft} by a less systematic method.
}
\end{remark}

%A graph~$G$ is called vertex-transitive if for any two vertices $x$ and $y$ there exists a graph automorphism~$\phi$ such that~$\phi(x)=y$.
%For vertex transitive graphs we have a more precise relation between~$\tin$ and~$\tmix$.

Throughout this article, unless mentioned otherwise,
whenever we consider a finite graph, we will always perform a lazy simple random walk on it.

Finally for all regular graphs, we show the following proposition in Section~\ref{sec:regular}.

\begin{proposition}\label{pro:regulargraphs}
Let $G$ be a connected regular graph on $n$ vertices. Then

{\rm(a)} $\thit \lesssim t_I^2$

{\rm(b)} $t_I\lesssim \sqrt{n} \left(\tunif\right)^{\frac{3}{4}}$.
\end{proposition}

\begin{remark}\rm{
We note that both bounds are sharp in the sense that there exist regular graphs attaining them. In particular, consider a random walk on a complete graph on $n$~vertices. Then $\tin\asymp \sqrt{n}$ and $\thit=n-1$. For a simple random walk on the cycle~$\Z_n$ we have $\thit\asymp n^2$ and~$\tunif\asymp n^2$.
}
\end{remark}

%As stated earlier, to the best of our knowledge the intersection time
%as defined above was first introduced in~\cite{interndraft} for finite
%graphs, although intersections of random walks on infinite lattices
%$\Z^d$ have been studied for a long time \cite{Lawlerinter}.

%In the finite case that we consider here,
The intersection time is related
to basic sampling questions~\cite{LSC14}, testing statistical properties
of distributions \cite{BatuFRSW13} and testing structural properties of
graphs, in particular expansion and conductance \cite{CzumajS10, KaleS11,
KalebS11}. Many of the approaches used in these works rely on
collision or intersections of random walks (or more generally, random
experiments), which is quite natural if one is interested in the
algorithms which work even in sublinear time (or space).  In this
context, it is particularly important to understand the relation between
these parameters and the expansion of the underlying graph, as done in our result which relates the mixing time to the
intersection time.

%Another possible domain where the intersection time may be useful is agent-based graph exploration via random walks, where the intersection time provides an estimate of the speed by which two agents communicate in a large, unknown network. As an example, in the ``white-board model'' where agents could leave messages at a node, the intersection time would be the minimum waiting time between the first time an agent starts dropping a message at every node until it is received by the other agent. [Here we may actually discuss the (hopefully equivalent?) variant of the intersection time, where the first walk has to reach a vertex which has been already visited by the other walk before. This could be also related to detecting invasions/epidemics/attacks etc.]

We further point out that there exists a seemingly related notion for single random walks, called self-intersection time. This time plays a key role in the context of finding the discrete logarithm using Markov chains \cite{discretelog}. However, we are not aware of any direct connection between this parameter and the intersection time of two random walks, as the self-intersection time will be just a constant for many natural classes of graphs.

\section{Intersection time for reversible Markov chains}\label{sec:reversible}

In this section we give the proof of Theorem~\ref{thm:upperboundmix}. We start by stating a result proved independently by Oliveira~\cite{mixhitroberto}, and Peres and Sousi~\cite{mixhit} that relates the total variation mixing time to the maximum hitting time of large sets for lazy reversible Markov chains.

\begin{theorem}[\cite{mixhitroberto}, \cite{mixhit}]\label{thm:mixhit}
Let $X$ be a lazy reversible Markov chain with stationary distribution $\pi$. Then we have
\[
\tmix \asymp \max_{x,A: \pi(A)\geq \frac{1}{8}} \estart{\tau_A}{x},
\]
where $\tau_A$ is the first hitting time of the set~$A$, i.e.\ $\tau_A=\inf\{t\geq 0: X_t\in A\}$.
\end{theorem}

For random walks on trees mixing times are equivalent to hitting times of the so-called ``central nodes''.
\begin{definition}\rm{
A node~$v$ of a tree~$T$ is called {\it{central}} if each component of $T-\{v\}$ has stationary probability at most $1/2$.
}
\end{definition}

\begin{theorem}[\cite{mixhit}]\label{thm:treetmix}
Let $X$ be a lazy weighted random walk on a tree $T$ and let~$v$ be a central node {\rm{(}}which always exists{\rm)}. Then
\[
\tmix \asymp \max_x\estart{\tau_v}{x},
\]
where~$\tau_v$ is the first hitting time of $v$.
\end{theorem}

Before proving Theorem~\ref{thm:upperboundmix} we introduce another notion
\[
t_I^* = \max_{x} \estart{\tau_I}{x,\pi}.
\]
Note the difference between $t_I^*$ and $\tin$ is that instead of maximizing over all starting points, in $\tin^*$ we start one chain from stationarity and maximize over the starting point of the other one.

\begin{proposition}\label{pro:titistar}
For all Markov chains we have
\[
\tin \asymp \tin^*.
\]
\end{proposition}

\begin{proof}[\bf Proof]
Obviously we have $\tin^*\leq \tin$, so we only need to prove that $\tin\lesssim \tin^*$.
To do so, we consider three independent chains, $X$, $Y$ and~$Z$ such that $X_0=x$, $Y_0=y$ and $Z_0\sim \pi$. We will denote by~$\tau_{I}^{X,Y}$ the first time that $X$ and $Y$ intersect and similarly for $\tau_I^{X,Z}$.

Let $t=6t_I^*$. It suffices to show that for all $x,y$ we have
\begin{align}\label{eq:maingoal}
\prstart{\tau_I^{X,Y}\leq 4t_I^*}{x,y} \geq c>0,
\end{align}
since then by performing independent experiments, we would get that $\tin \lesssim \tin^*$.
For all $0\leq k\leq t$ we define $$
M_k = \prcond{Y[0,4t] \cap Z[2t,3t]= \emptyset}{Z_0,\ldots, Z_k}{y,\pi} = \prcond{Y[0,4t] \cap Z[2t,3t]= \emptyset}{Z_k}{y,\pi} ,
$$
where the last equality follows from the Markov property. Then clearly~$M$ is a martingale.
By Doob's maximal inequality we get
\begin{align*}
\prstart{\max_{0\leq k\leq t} M_k\geq \frac{3}{4}}{y,\pi} &\leq \frac{4}{3}\cdot \prstart{Y[0,4t] \cap Z[2t,3t]= \emptyset}{y,\pi} \\
&\leq \frac{4}{3} \cdot \prstart{Y[2t,3t]\cap Z[2t,3t] =\emptyset}{y,\pi} \\&
\leq \frac{4}{3}\cdot \max_{x} \prstart{\tau_I\geq t}{x,\pi}
\leq \frac{4}{3}\cdot \frac{\max_{x} \estart{\tau_I}{x,\pi}}{t} = \frac{4t_I^*}{3t} = \frac{2}{9},
\end{align*}
where in the final inequality we used Markov's inequality. Next we define
\[
G=\left\{ \max_{0\leq k\leq t} M_k\leq  \frac{3}{4} \quad \text{and}\quad \tau_I^{X,Z} \leq  t \right\}.
\]
By the union bound and Markov's inequality we obtain
\begin{align}\label{eq:unionboundg}
\prstart{G^c}{x,y,\pi} \leq \frac{2}{9} + \frac{1}{6} = \frac{7}{18}.
\end{align}
If $\sigma=\inf\{k: X_k\in Z[0,t]\}\wedge t$ and $B=\{w: \prstart{Y[0,4t]\cap Z[t,3t] \neq \emptyset}{y,w} \geq 1/4\}$, then we have
\begin{align*}
\prstart{\tau_I^{X,Y} \leq 5t}{x,y} \geq \prstart{\tau_I^{X,Y} \leq 5t, G}{x,y,\pi} = \sum_{w\in B} \prstart{\tau_I^{X,Y} \leq 5t, G, X_{\sigma} =w}{x,y,\pi}.
\end{align*}
For the last equality we note that on $G$ if $X_\sigma=w\notin B$, then $\exists \,\ell\leq t$ such that $Z_\ell =w\notin B$, and hence on this event we have
\begin{align*}
\prcond{Y[0,4t]\cap Z[2t,3t]\neq \emptyset}{Z_\ell}{y,\pi} &=\prcond{Y[0,4t]\cap Z[2t,3t]\neq \emptyset}{Z_\ell=w}{y,\pi}  \\
&= \prstart{Y[0,4t]\cap Z[2t-\ell,3t-\ell]\neq \emptyset}{y,w} \\
&\leq \prstart{Y[0,4t]\cap Z[t,3t]\neq \emptyset}{y,w} <\frac{1}{4} \Longrightarrow G^c.
\end{align*}
We now deduce
\begin{align*}
\prstart{\tau_I^{X,Y} \leq 5t}{x,y} &\geq \sum_{w\in B} \prcond{\tau_{I}^{X,Y}\leq 5t}{G,X_\sigma =w}{x,y,\pi} \prcond{X_\sigma =w}{G}{x,\pi} \prstart{G}{x,\pi}\\
& \geq \sum_{w\in B} \prstart{\tau_I^{X,Y}\leq 4t}{w,y} \prcond{X_\sigma =w}{G}{x,\pi} \prstart{G}{x,\pi} \\
&=\sum_{w\in B} \prstart{Y[0,4t]\cap X[0,4t]\neq \emptyset}{w,y} \prcond{X_\sigma =w}{G}{x,\pi} \prstart{G}{x,\pi} \\
&\geq  \sum_{w\in B} \prstart{Y[0,4t]\cap Z[t,3t]\neq \emptyset}{y,w} \prcond{X_\sigma =w}{G}{x,\pi} \prstart{G}{x,\pi}\geq \frac{1}{4} \cdot \frac{11}{18},
\end{align*}
The first inequality follows from the Markov property, since  the events $G$ and $\{X_\sigma=w\}$ only depend on the paths of the chains~$X$ and~$Z$ up to time $t$. The last inequality follows from~\eqref{eq:unionboundg} and the definition of the set~$B$ and this concludes the proof of~\eqref{eq:maingoal}.
\end{proof}

\begin{proposition}\label{pro:tmixtistar}
For all lazy reversible Markov chains we have
\[
\tmax \lesssim t_I^*.
\]
\end{proposition}

\begin{proof}[\bf Proof]

The proof of this proposition is similar and simpler than the proof of Proposition~\ref{pro:titistar}. We include it here for the sake of completeness.

Let $X$ and $Y$ be two independent lazy Markov chains such that $X_0=x$ and $Y_0\sim \pi$.
Let~$A$ be a set with $\pi(A)\geq 1/8$ and define
\[
\tau_A = \inf\{t\geq 0: X_t\in A\}.
\]
 Then we claim that for all $x$ we have
\begin{align}\label{eq:firstclaim}
\prstart{\tau_A \leq 12 t_I^*}{x} \geq c>0.
\end{align}
First of all by Markov's inequality we immediately get
\begin{align}\label{eq:markov}
\prstart{\tau_I\geq 6t_I^*}{x,\pi} \leq \frac{1}{6}.
\end{align}
Let $t=6t_I^*$ and for $0\leq k\leq t$ we let
\begin{align*}
M_k = \prcond{Y_t \in A^c}{Y_0,\ldots, Y_k}{\pi} = \prcond{Y_t \in A^c}{Y_k}{\pi},
\end{align*}
where the second equality follows by the Markov property. It follows from the definition of $M$ that it is a martingale, and hence applying Doob's maximal inequality, we immediately obtain
\begin{align}\label{eq:doob}
\prstart{\max_{0\leq k\leq t} M_k\geq \frac{3}{4}}{\pi} \leq \frac{4}{3}\cdot\estart{M_t}{\pi}= \frac{4}{3} \cdot \prstart{Y_t\in A^c}{\pi} \leq \frac{1}{2},
\end{align}
since $\pi(A) \geq 1/8$. We now let
\[
G = \left\{\max_{0\leq k\leq t} M_k \leq \frac{3}{4} \quad \text{and} \quad \tau_I \leq t\right\}.
\]
By the union bound and using~\eqref{eq:doob} and~\eqref{eq:markov} we obtain
\begin{align*}
\prstart{G^c}{x,\pi} \leq \frac{1}{2} + \frac{1}{6} = \frac{2}{3}.
\end{align*}
Letting $\sigma = \min\{k: X_k\in Y[0,t]\}\wedge t$ and $B=\{z: \prstart{\tau_A\leq t}{z} \geq 1/4\}$, we now get
\begin{align*}
\prstart{\tau_A\leq 2t}{x} \geq \prstart{\tau_A\leq 2t, G}{x,\pi} = \sum_{z\in B} \prstart{\tau_A\leq 2t, G, X_\sigma = z}{x,\pi}.
\end{align*}
The last equality is justified, since if $X_\sigma=z\notin B$, then $\exists k$ such that $Y_k=z\notin B$, and hence on this event we have
\begin{align*}
 \prcond{Y_t\in A}{Y_k}{\pi} <
\frac{1}{4} \Rightarrow \max_{0\leq k\leq t} M_k > \frac{3}{4} \Rightarrow G^c.
\end{align*}
Therefore we deduce that
\begin{align*}
\prstart{\tau_A\leq 2t}{x} &\geq \sum_{z\in B} \prcond{\tau_A\leq 2t}{G, X_\sigma =z}{x,\pi} \prcond{X_\sigma=z}{G}{x,\pi}\prstart{G}{x,\pi}
\\& \geq \sum_{z\in B} \prstart{\tau_A\leq t}{z} \prcond{X_\sigma=z}{G}{x,\pi} \prstart{G}{x,\pi} \geq \frac{1}{4} \cdot \frac{1}{3} = \frac{1}{12},
\end{align*}
where the second inequality follows by the Markov property, since the events~$G$ and $\{X_\sigma =z\}$ only depend on the paths of the chains up to time $t$.
This concludes the proof of~\eqref{eq:firstclaim} and by performing independent geometric experiments, we finally get that
\begin{align*}
\max_x \estart{\tau_A}{x} \lesssim t_I^*.
\end{align*}
Since $A$ was an arbitrary set with $\pi(A)\geq 1/8$, we get
\begin{align*}
\tmax \lesssim t_I^*
\end{align*}
and this finishes the proof.
\end{proof}

\begin{proof}[\bf Proof of Theorem~\ref{thm:upperboundmix}]

Propositions~\ref{pro:titistar} and~\ref{pro:tmixtistar} immediately give that for all  Markov chains we have
\[
\tmax \lesssim \tin
\]
and this finishes the proof.
\end{proof}

\begin{proof}[\bf Proof of Corollary~\ref{cor:tmix}]

Using the equivalence between mixing times and hitting times of large sets for reversible chains by Theorem~\ref{thm:mixhit} combined with the statement of Theorem~\ref{thm:upperboundmix} shows that
\[
\tmix\lesssim \tin.
\]
It remains to prove that for trees the two quantities, $\tmix$ and $\tin$, are equivalent. Since $\tmix\lesssim \tin$ for all reversible Markov chains, we only need to show that $\tin\lesssim \tmix$. Let~$v$ be a central node. Then if we wait until both chains $X$ and~$Y$ hit $v$, this will give an upper bound on their intersection time, and hence
\[
\estart{\tau_I}{x,y} \leq \estart{\tau_v^X}{x} + \estart{\tau_v^Y}{y} \leq 2t_v.
\]
Now Theorem~\ref{thm:treetmix} finishes the proof.
\end{proof}

\section{Intersection time for transitive chains}\label{sec:transitive}

In this section we prove Theorem~\ref{thm:transitive}. We start by showing that for transitive chains instead of considering one or two worst starting points, both chains can start from stationarity. In particular, we have the following.

\begin{lemma}\label{lem:stationary}
Let $X$ be a transitive and reversible chain on a finite state space. Then
\[
t_I\asymp \estart{\tau_I}{\pi,\pi}.
\]
\end{lemma}

\begin{proof}[\bf Proof]

From Proposition~\ref{pro:titistar} we have that for all reversible chains
\begin{align*}
t_I\asymp \max_x \estart{\tau_I}{x,\pi}.
\end{align*}
By transitivity it follows that $\estart{\tau_I}{x,\pi}$ is independent of $x$. Therefore, averaging over all $x$ in the state space proves the lemma.
\end{proof}

For a transitive chain we define for all $t>0$
\[
g_t(x,z) = \sum_{j=0}^{t} p_j(x,z) \quad \text{and}\quad Q_t = \sum_z g_t^2(x,z).
\]
Note that by transitivity $Q_t$ does not depend on $x$.

The next lemma gives a control on the first and second moment of the number of intersections of two independent transitive chains. It will be used in the proof of Theorem~\ref{thm:transitive}. In this form it appeared in~\cite{LPSinters}, but the idea goes back to Le-Gall and Rosen~\cite[Lemma~3.1]{LeGallRosen}. We include the proof here for the reader's convenience.

\begin{lemma}\label{lem:samestartingpoint}
Let $X$ and $Y$ be two independent transitive chains and $I_t=\sum_{i=0}^{t}\sum_{j=0}^{t}\1(X_i=Y_j)$ count the number of intersections up to time~$t$. Then for all $x$ we have
\[
\estart{I_t}{x,x} =Q_t \quad \text{and}\quad \estart{I_t^2}{x,x} \leq 4 Q_t^2.
\]
\end{lemma}

\begin{proof}[\bf Proof]

For the first moment of the number of intersections we have
\begin{align*}
\estart{I_t}{x,x} = \sum_{i=0}^{t} \sum_{j=0}^{t} \prstart{X_i=Y_j}{x,x} = \sum_z \sum_{i=0}^{t}\sum_{j=0}^{t} \prstart{X_i=z}{x} \prstart{Y_j=z}{x} = \sum_z g_t^2(x,z)=Q_t.
\end{align*}
For the second moment of $I_t$ we have
\begin{align*}
\estart{I_t^2}{x,x}  &= \sum_{i,j,\ell,m=0}^{t}  \prstart{X_i=Y_j, X_\ell = Y_m}{x,y} = \sum_{z,w} \sum_{i,j,\ell,m=0}^{t} \prstart{X_i=z, X_\ell=w}{x} \prstart{Y_j=z,Y_m=w}{y} \\
&\leq  \sum_{z,w} (g_t(x,z) g_t(z,w) +g_t(x,w)g_t(w,z))^2 \\&\leq 2\sum_{z,w} (g_t^2(x,z)g_t^2(z,w) + g_t^2(x,w)g_t^2(w,z))=4Q_t^2.
\end{align*}
For the second inequality we used $(a+b)^2\leq 2(a^2 + b^2)$ and for the last one we used transitivity.
\end{proof}

\begin{lemma}\label{lem:prelim}
Let $X$ be a transitive chain on $n$ states starting from $x$ and $S_t(x) = \sum_{j=0}^{t} g_t(x,X_j)$.
Then
\[
\prstart{S_t(x)\geq \frac{Q_t}{2}}{x} \geq \frac{1}{16}.
\]
\end{lemma}

\begin{proof}[\bf Proof]
Let $X$ and $Y$ be two independent copies of the chain starting from $x$. We write
\[
I_t=\sum_{j=0}^{t}\sum_{\ell=0}^{t} \1(X_j=Y_\ell)
\]
for the total number of intersections up to time $t$. We now observe that
\[
S_t(x) = \escond{I_t}{X_0,\ldots, X_t}{x},
\]
and hence we get
\[
\estart{S_t(x)}{x} = \estart{I_t}{x,x} =Q_t \quad \text{and}\quad
\estart{S_t^2(x)}{x} \leq \estart{I_t^2}{x,x}.
\]
From Lemma~\ref{lem:samestartingpoint} we now obtain
\[
\estart{S_t^2(x)}{x} \leq \estart{I_t^2}{x,x}\leq 4 (\estart{I_t}{x,x})^2 = 4Q_t^2.
\]
Applying the second moment method finally gives
\[
\prstart{S_t(x) \geq \frac{Q_t}{2}}{x} \geq \frac{1}{4} \cdot \frac{(\estart{S_t(x)}{x})^2}{ \estart{S_t^2(x)}{x}} \geq \frac{1}{16}
\]
and this concludes the proof.
\end{proof}

The following proposition is the key ingredient of the proof of Theorem~\ref{thm:transitive}.
We now explain the key idea behind the proof which was used in~\cite[Theorem~5.1]{HagPerSteif}. We define a set of good points on the path of the chain~$X$ and show that conditional on~$X$ and~$Y$ intersecting before time $t$, then they intersect at a good point with constant probability .

\begin{proposition}\label{pro:interstrans}
Let $X$ and $Y$ be two independent copies of a transitive chain on $n$ states started from stationarity. Let $I_t$ denote the number of intersections of $X$ and $Y$ up to time $t$. Then
\[
\frac{(t+1)^2}{4nQ_t} \leq \prstart{I_t>0}{\pi,\pi}  \leq \frac{2^7(t+1)^2}{nQ_t}.
\]
\end{proposition}

\begin{proof}[\bf Proof]
For all $t$ using the independence between $X$ and $Y$ we get
\begin{align*}
\estart{I_t}{\pi,\pi} = \sum_{z} \sum_{i,j=0}^{t}  \prstart{X_i=z,Y_j=z}{\pi,\pi} =\frac{(t+1)^2}{n}.
\end{align*}
For the second moment we have
\begin{align}\label{eq:firstmoment}
\estart{I_t^2}{\pi,\pi} &=\sum_{i,j,\ell,m=0}^{t} \sum_{z,w} \prstart{X_i=z, X_j=w}{\pi}\prstart{Y_\ell=z, Y_m=w}{\pi}\\
&\leq \frac{(t+1)^2}{n^2} \sum_{z,w} \left(g_t(z,w) + g_t(w,z) \right)^2 = \frac{4(t+1)^2}{n} Q_t,
\end{align}
where for the last equality we used transitivity.
Using the second moment method we obtain
\[
\prstart{I_t>0}{\pi,\pi} \geq \frac{(t+1)^2}{4nQ_t}.
\]
We now turn to prove the upper bound. For every $x=(x_0,\ldots, x_{2t})$ we define the set
\[
\Gamma_t(x) = \left\{r\leq t: \sum_{j=0}^{t}g_t(x_r,x_{r+j}) \geq \frac{Q_t}{2}\right\}.
\]
%This is a random set that depends on $X_0,X_1,\ldots, X_{2t}$.
By Lemma~\ref{lem:prelim}
we have that for all $r\leq t$ and all $z$
\begin{align}\label{eq:keyeq}
\prstart{r\in \Gamma_t(X)}{z} \geq \frac{1}{16},
\end{align}
where to simplify notation we write $\Gamma_t(X)$ for the random set $\Gamma_t((X_s)_{s\leq 2t})$
Next we define
\[
\tau=\min\{j\in [0,t]: \, X_j\in \{Y_0,\ldots, Y_t\} \},
\]
and $\tau=\infty$ if the above set is empty. Conditioned on $(Y_s)_{s\leq t}$, we see that $\tau$ is a stopping time for~$X$.  Thus using also~\eqref{eq:keyeq} we get that~$\tau$ satisfies
\[
\prcond{\tau\in \Gamma_t(X)}{\tau<\infty}{\pi,\pi} \geq \frac{1}{16}.
\]
Therefore
\begin{align}\label{eq:important}
\prstart{I_t>0}{\pi,\pi} = \prstart{\tau<\infty}{\pi,\pi} \leq 16\cdot \prstart{\tau\in \Gamma_t(X)}{\pi,\pi}.
\end{align}
It now remains to bound $\prstart{\tau\in \Gamma_t(X)}{\pi,\pi}$. We define $\sigma=\min\{\ell\in [0,t]: \, Y_\ell \in \cup_{r\in \Gamma_t(X)} X_r\}$
and we note that
\begin{align}\label{eq:tausigma}
\prstart{\tau\in \Gamma_t(X)}{\pi,\pi} \leq \prstart{\sigma\in [0,t]}{\pi,\pi}.
\end{align}
Writing $A_k=\{Y_\sigma=X_k, \, k\in \Gamma_t, \, k\,\text{is minimal},\, \sigma\in [0,t]\}$ for all $k\leq t$ we now have
\begin{align}\label{eq:long1}
&\escond{I_{2t}}{\sigma\in [0,t]}{\pi,\pi} = \sum_{k=0}^{t} \escond{I_{2t}}{A_k}{\pi,\pi}\prcond{A_k}{\sigma\in [0,t]}{\pi,\pi}.
\end{align}
For every $k\leq t$ we obtain
\begin{align*}
\escond{I_{2t}}{A_k}{\pi,\pi} &\geq  \sum_{\substack{x=(x_0,\ldots,x_{2t})\\\text{s.t.} \,k \in \Gamma_t(x)}}\escond{\sum_{i,j=0}^{t}\1(Y_{\sigma+i} = X_{k+j})}{(X_s)_{s\leq 2t}=x,A_k}{\pi,\pi}\prcond{(X_s)_{s\leq 2t}=x}{A_k}{\pi,\pi} \\
 &= \sum_{\substack{x=(x_0,\ldots,x_{2t})\\ \text{s.t.}\,k\in \Gamma_t(x)}} \sum_{j=0}^{t}g_t(x_k,x_{k+j}) \prcond{(X_s)_{s\leq 2t}=x}{A_k}{\pi,\pi} \geq \frac{Q_t}{2}.
\end{align*}
Substituting the above lower bound into~\eqref{eq:long1} we deduce
\begin{align*}
\escond{I_{2t}}{\sigma\in [0,t]}{\pi,\pi} \geq \frac{Q_t}{2}.
\end{align*}
Using~\eqref{eq:firstmoment} and the above bound we finally get
\begin{align*}
\prstart{\sigma\in [0,t]}{\pi,\pi} \leq \frac{\estart{I_{2t}}{\pi,\pi}}{\escond{I_{2t}}{\sigma\in [0,t]}{\pi,\pi}} \leq \frac{(2t+1)^2/n}{Q_t/2} \leq  \frac{2^3 (t+1)^2}{nQ_t}.
\end{align*}
This in conjunction with~\eqref{eq:important} and~\eqref{eq:tausigma} gives
\[
\prstart{I_t>0}{\pi,\pi} \leq \frac{2^7(t+1)^2}{nQ_t},
\]
and this concludes the proof of the upper bound.
\end{proof}

The following lemma follows by the spectral theorem and will be used for the upper bound in the proof of Theorem~\ref{thm:transitive}. Combined with the statement of Theorem~\ref{thm:transitive} it gives that for transitive and reversible chains $\tunif\lesssim \tin$, which is an improvement over Corollary~\ref{cor:tmix} which gives $\tmix\lesssim \tin$. Note that this is not true in general, if the chain is not transitive. Take for instance two cliques of sizes~$\sqrt{n}$ and~$n$ connected by a single edge.

\begin{lemma}\label{lem:tunifbound}
Let $X$ be a reversible, transitive and lazy chain on $n$ states and $(\lambda_j)_j$ are the corresponding non-unit eigenvalues. Then
\[
\tunif \leq 2\sqrt{Q},
\]
where $Q=\sum_{k=2}^{n}(1-\lambda_k)^{-2}$.
\end{lemma}

\begin{proof}[\bf Proof]

We start by noting that for a transitive, reversible and lazy chain the uniform mixing time is given by
\[
\tunif = \min\left\{t\geq 0: p_t(x,x) \leq \frac{5}{4n}\right\}.
\]
See for instance~\cite[equation~(16)]{evolvingsets} or~\cite[Proposition~A.1]{Shayan}.
By the spectral theorem and using transitivity of $X$ we have
\[
p_t(x,x) =\frac{1}{n}\cdot \sum_{k=1}^{t} \lambda_k^t = \frac{1}{n}+ \frac{1}{n}\cdot \sum_{k=2}^{n}\lambda_k^t.
\]
Therefore $\tunif=\min\{t:\, \sum_{k=2}^{n} \lambda_k^t\leq 1/4\}$.
We now set $\epsilon_j=1-\lambda_j$ for all $j$. Since the chain is lazy, it follows that $\epsilon_j\in [0,1]$ for all $j$.
So we now need to show
\begin{align}\label{eq:tuupper}
\sum_{k=2}^{n}(1-\epsilon_k)^{2\sqrt{\sum_{j=2}^{n}\epsilon_j^{-2}}} \leq \frac{1}{4}.
\end{align}
 In order to prove~\eqref{eq:tuupper} it suffices to show
\begin{align*}
\sum_{k=2}^{n}\exp\left(-2\epsilon_k\cdot  \sqrt{\sum_{j=2}^{n}\epsilon_j^{-2}} \right) \leq \frac{1}{4}.
\end{align*}
Writing $r_k= \epsilon_k\cdot \sqrt{\sum_{j=2}^{n}\epsilon_j^{-2}}$, we get $r_k\geq 1$ and $\sum_{k=2}^{n}r_k^{-2} =1$. Since $e^{r}\geq r^2$ for all $r\geq 0$, we finally deduce
\begin{align*}
\sum_{k=2}^{n} e^{-2r_k} \leq \frac{1}{4}\cdot \sum_{k=2}^{n} r_k^{-2} = \frac{1}{4}
\end{align*}
and this finishes the proof.
\end{proof}

We are now ready to give the proof of Theorem~\ref{thm:transitive}.

\begin{proof}[\bf Proof of Theorem~\ref{thm:transitive}]

Since the chain is reversible and transitive, it follows that for any state~$x$ we have
\[
Q_t = \sum_{i=0}^{t} \sum_{j=0}^{t} p_{i+j}(x,x).
\]
Using the spectral theorem together with transitivity, we obtain
\begin{align}\label{eq:qtexpr}
Q_t = \frac{1}{n} \cdot \sum_{k=1}^{n} \sum_{i,j=0}^{t} \lambda_k^{i+j} = \frac{(t+1)^2}{n} + \frac{1}{n} \cdot \sum_{k=2}^{n} \frac{(1-\lambda_k^{t+1})^2}{(1-\lambda_k)^2}.
\end{align}

For $t\geq t_{{\rm{rel}}} = (1-\lambda_2)^{-1}$ we get
\begin{align*}
\left(1-\lambda_2^{t+1}\right)^2 \geq 1-2\lambda_2^{t+1} \geq 1-2\lambda_2^t \geq 1-\frac{2}{e}.
\end{align*}
Since for all $j\geq 2$ we have $\lambda_j\leq \lambda_2$ using the above inequality we obtain for all $j\geq 2$ and $t\geq t_{{\rm rel}}$
\begin{align*}
\left(1-\lambda_j^{t+1}\right)^2 \geq 1-\frac{2}{e}.
\end{align*}
Therefore for all $t\geq t_{\rm rel}$ we deduce
\begin{align}\label{eq:lowerqt}
Q_t \geq \frac{(t+1)^2}{n} + \left(1-\frac{2}{e} \right)\cdot \frac{Q}{n}.
\end{align}
Using~\eqref{eq:lowerqt} together with Proposition~\ref{pro:interstrans} now gives for $t\geq t_{\rm rel}$
\begin{align}\label{eq:itbound}
\prstart{\tau_I\leq t}{\pi,\pi} \leq \frac{2^7(t+1)^2}{(t+1)^2 + \left(1-\frac{2}{e}\right) Q}.
\end{align}
We now claim that $\tin\gtrsim \sqrt{Q}$. Let $C_1$ be a large constant to be specified later. If $\sqrt{Q} \leq C_1t_{\rm rel}$, then the claim follows from Corollary~\ref{cor:tmix}. So we may assume that $\sqrt{Q} \geq C_1t_{\rm rel}$. Setting $t=C\sqrt{Q}\geq t_{\rm rel}$ for a constant $C\geq 1/C_1$ to be determined we get
\begin{align*}
\prstart{\tau_I\leq t}{\pi,\pi} \leq 2^7\cdot \frac{C^2Q}{C^2Q + \left(1-\frac{2}{e}\right)Q}.
\end{align*}
If we take $C$ so that $C^2= (1-2/e)/2^8 $ and we choose $C_1=(1-2/e)^{-1/2}\cdot 2^4$, then from the above we obtain
\[
\prstart{\tau_I\leq t}{\pi,\pi}\leq \frac{1}{2}
\]
and this proves the claim that $\tin \gtrsim \sqrt{Q}$. 
It remains to show that $\tin \lesssim \sqrt{Q}$. It suffices to show that there are positive constants $c_1$ and $c_2$ such that for all $x,y$ we have
\begin{align}\label{eq:progoal}
\prstart{\tau_I\leq c_1\sqrt{Q}}{x,y} \geq c_2.
\end{align}
Indeed, by then performing independent experiments, we would get that $\tin \lesssim \sqrt{Q}$. 
From~\eqref{eq:qtexpr} we immediately get
\begin{align}\label{eq:upperqt}
Q_t \leq \frac{(t+1)^2}{n} +  \frac{Q}{n}.
\end{align}
This together with Proposition~\ref{pro:interstrans} gives that for all $t$ we have
\begin{align}\label{eq:lowerboundprob}
\prstart{\tau_I\leq t}{\pi,\pi}\geq \frac{1}{4}\cdot \frac{(t+1)^2}{(t+1)^2+ Q}.
\end{align}
Taking~$t=\sqrt{Q}$ in~\eqref{eq:lowerboundprob} gives
\begin{align}\label{eq:1/8}
\prstart{\tau_I\leq \sqrt{Q}}{\pi,\pi} \geq \frac{1}{8}.
\end{align}
From Lemma~\ref{lem:tunifbound} we have $\tunif\leq 2\sqrt{Q}$. Setting $s=2\sqrt{Q}$ we now have for all $x,y$
\begin{align*}
\prstart{\tau_I\leq s+\sqrt{Q}}{x,y} &\geq \prstart{X[s,s+\sqrt{Q}]\cap Y[s,s+\sqrt{Q}] \neq \emptyset}{x,y}\\ &=
\sum_{x',y'} p_{s}(x,x') p_{s}(y,y') \prstart{\tau_I\leq \sqrt{Q}}{x',y'} \\&
\geq \frac{9}{16}\cdot \sum_{x',y'} \pi(x')\pi(y')\prstart{\tau_I\leq \sqrt{Q}}{x',y'}
\\&
\geq \frac{9}{16} \prstart{\tau_I\leq \sqrt{Q}}{\pi,\pi} \geq \frac{9}{128},
\end{align*}
where for the last inequality we used~\eqref{eq:1/8}. This proves~\eqref{eq:progoal}. Finally, from~\eqref{eq:lowerqt}, \eqref{eq:upperqt} and since $\tunif\leq 2\sqrt{Q}$ by Lemma~\ref{lem:tunifbound} we obtain
\[
Q_{\tunif} = \sum_{i,j=0}^{\tunif} p_{i+j}(x,x) \asymp \frac{Q}{n}.
\]
and this concludes the proof of the theorem.
\end{proof}

\section{Intersection time for regular graphs}\label{sec:regular}

In this section we prove Proposition~\ref{pro:regulargraphs} which gives bounds on the intersection time for random walks on regular graphs. We start by proving the first part of Proposition~\ref{pro:regulargraphs}.

\begin{proof}[{\bf Proof of Proposition~\ref{pro:regulargraphs}} {\rm(part (a))}]

Let $t = \sqrt{\thit}/2$ and $y$ be such that $\thit = \max_x \estart{\tau_y}{x}$, where we recall that $\tau_y$ stands for the first hitting time of~$y$ by a simple random walk on~$G$. Then there exists~$z$ such that
\begin{align}\label{eq:yz}
\prstart{\tau_y \leq t}{z} \leq \frac{2t}{\thit},
\end{align}
since otherwise we would get $\max_x\estart{\tau_y}{x} \leq \thit/2$, which contradicts the choice of $y$.
Let~$Y$ and~$Z$ be two independent random walks started from~$y$ and~$z$ respectively. Then by the union bound we get
\begin{align}\label{eq:unionbound}
\prstart{\tau_I <\frac{t}{2}}{y,z} \leq \sum_{k=1}^{t/2} \prstart{\tau^Z_{Y_k}<\frac{t}{2}}{y,z},
\end{align}
where $\tau^Z_{x}$ stands for the first hitting time of $x$ by the random walk~$Z$. We note that by reversibility and regularity we have
\begin{align}\label{eq:bigeq}
\nonumber\prstart{\tau^Z_{Y_k}<\frac{t}{2}}{y,z} &= \sum_w \prstart{\tau^Z_w<\frac{t}{2}, Y_k =w}{y,z} = \sum_w \prstart{\tau^Z_w<\frac{t}{2}}{z} \prstart{Y_k=w}{y} \\
&= \sum_w \prstart{\tau^Z_w<\frac{t}{2}}{z} \prstart{Y_k=y}{w} = \sum_w \prstart{\tau^Z_w<\frac{t}{2}, Y_k = y}{w,z}.
\end{align}
Consider now a third walk~$X$ such that $X_s=Z_s$ for all~$s\leq \tau_{Y_0}^Z$ and $X_s = Y_{s-\tau_{Y_0}^Z}$ for $\tau_{Y_0}^Z\leq s\leq \tau_{Y_0}^Z+k$.
We now obtain
\begin{align*}
\sum_w\prstart{\tau^Z_w<\frac{t}{2}, Y_k = y}{w,z}  &= \sum_w\sum_{s<\frac{t}{2}} \prstart{\tau_w^X=s, X_{s+k}=y}{z}=  \sum_{s<\frac{t}{2}}\sum_w \prstart{\tau_w^X=s, X_{s+k}=y}{z} \\
&\leq \sum_{s<\frac{t}{2}} \prstart{X_{s+k}=y}{z}  = \prstart{\tau^X_y<\frac{t}{2}+k}{z} \leq \prstart{\tau^X_y<t}{z} \leq \frac{2t}{\thit},
\end{align*}
where the penultimate inequality follows since $k\leq t/2$ and the final inequality from the choice of~$z$ and~$y$ in~\eqref{eq:yz}. Combining~\eqref{eq:unionbound} with~\eqref{eq:bigeq} and the above inequality, we therefore conclude that
\begin{align*}
\prstart{\tau_I <\frac{t}{2}}{y,z} \leq \frac{t^2}{\thit} = \frac{1}{4}
\end{align*}
and this gives
\[
\estart{\tau_I}{z,y} \gtrsim \sqrt{\thit}.
\]
This finally implies the desired inequality, i.e.\ $\thit\lesssim t_I^2$.
\end{proof}

Before proving part (b) of Proposition we state a result about return probabilities for random walks on regular graphs.
Its proof follows for instance from~\cite[Proposition~6.16, Chapter~6]{AldFill}. We also state the Cauchy-Schwarz inequality for the transition probabilities for the sake of completeness. For a proof we refer the reader to~\cite[Lemma~3.20, Chapter~3]{AldFill}.

\begin{lemma}\label{lem:returnpro}
Let $G$ be a regular graph on $n$ vertices and $t\leq n^2$. Then for all vertices $x$ the return probability to $x$ satisfies
\[
P^t(x,x) \lesssim \frac{1}{\sqrt{t}}.
\]
\end{lemma}

\begin{lemma}\label{lem:cauchy}
Let $X$ be a lazy reversible Markov chain with transition matrix $P$ and stationary distribution~$\pi$. Then for all $x,y$ we have
\[
\frac{P^t(x,y)}{\pi(y)} \leq \sqrt{\frac{P^t(x,x)}{\pi(x)}\cdot  \frac{P^t(y,y)}{\pi(y)}}.
\]
In particular, if $X$ is a lazy simple random walk on a regular graph~$G$, then
\[
P^t(x,y) \leq \sqrt{P^t(x,x)\cdot  P^t(y,y)}.
\]
\end{lemma}

\begin{proof}[{\bf Proof of Proposition~\ref{pro:regulargraphs}} \rm{(part (b))}]

For this proof we assume that $X$ and $Y$ are lazy simple random walks on the graph $G$. Clearly, this only changes the intersection time by a multiplicative constant.

Let $t=c\sqrt{n} \left(\tunif\right)^{\frac{3}{4}}$ for a constant $c$ to be chosen later.
We define $I_t$ to be the total number of intersections of $X$ and $Y$ up to time $t$. We are going to use the second moment method, so we first have to calculate the first and second moments of~$I_t$.

For the first moment we have
\begin{align*}
\estart{I_t}{x,y} &= \sum_{v} \sum_{i,j=0}^{t} \prstart{X_i=Y_j = v}{x,y} = \sum_v \sum_{i,j=0}^{t} p_i(x,v) p_{j}(y,v) \\
&= \sum_{i,j=0}^{t}\sum_v p_i(x,v) p_j(v,y) = \sum_{i,j=0}^{t}  p_{i+j}(x,y) =A,
\end{align*}
where the third equality follows from reversibility and the regularity of the graph. For the second moment we have
\begin{align}\label{eq:long}
\nonumber&\estart{I_t^2}{x,y} = \sum_{v,w}\sum_{i,j,k,\ell=0}^{t}  \prstart{X_i = Y_j =v, X_k = Y_\ell=w}{x,y}\\
\nonumber &=\sum_{v,w} \sum_{i,j,k,\ell=0}^{t} \prstart{X_i=v, X_k=w}{x} \prstart{Y_j=v, Y_\ell =w }{y} \\
\nonumber&=  2\sum_{v,w}\sum_{\substack{(i,k)\\i\geq k}}\sum_{\substack{(j,\ell)\\ j\geq \ell}} p_k(x,w)p_{i-k}(w,v) p_{\ell}(y,w)p_{j-\ell}(w,v) \\
&\quad+2\sum_{v,w}\sum_{\substack{(i,k)\\i\geq k}}\sum_{\substack{(j,\ell)\\ j\leq \ell}} p_k(x,w)p_{i-k}(w,v) p_{j}(y,v)p_{\ell-j}(v,w)= 2\Sigma_1 + 2\Sigma_2.
\end{align}

We now treat each of the two sums~$\Sigma_1$ and $\Sigma_2$ appearing in~\eqref{eq:long} separately. For the first sum~$\Sigma_1$ using again reversibility and regularity of the graph we obtain by summing over~$v$ first that it is equal to
\begin{equation}\label{eq:sigma1}
\begin{split}
\Sigma_1&=\sum_w \sum_{\substack{(i,k)\\i\geq k}}\sum_{\substack{(j,\ell)\\ j\geq \ell}} p_k(x,w) p_{(i+j)-(k+\ell)}(w,w)p_\ell(y,w) \\
&\leq \sum_w\sum_{k,\ell} p_k(x,w) p_\ell(y,w)\sum_{\substack{i\leq t\\j\leq t}} \  p_{i+j}(w,w).
\end{split}
\end{equation}
Using Lemma~\ref{lem:returnpro} we obtain that
\begin{align*}
\sum_{i\leq t} \sum_{j\leq t} p_{i+j}(w,w) &= \sum_{i+j\leq \tunif} p_{i+j}(w,w) + \sum_{\substack{i+j>\tunif\\ i, j\leq t}} p_{i+j}(w,w) \\&\lesssim \sum_{i+j\leq \tunif} \frac{1}{\sqrt{i+j}} + \frac{t^2}{n}
\lesssim \left(\tunif\right)^{\frac{3}{2}} + \frac{t^2}{n}\asymp \left(\tunif \right)^{\frac{3}{2}},
\end{align*}
where the last step follows, since we took $t = c\sqrt{n}\left(\tunif\right)^{\frac{3}{4}}$. Substituting this bound to~\eqref{eq:sigma1} and using reversibility again we deduce
\begin{align*}
\Sigma_1\lesssim (\tunif)^{\frac{3}{2}}\sum_w\sum_{\ell,k\leq t} p_\ell(w,y) p_k(x,w) = (\tunif)^{\frac{3}{2}}\sum_{\ell,k\leq t} p_{k+\ell}(x,y) = (\tunif)^{\frac{3}{2}} \cdot A.
\end{align*}
For the second sum~$\Sigma_2$ appearing in~\eqref{eq:long} we get
\begin{align*}
\Sigma_2 &\lesssim \sum_{v,w}\sum_{\substack{(i,k)\\i\geq k}} \sum_{\substack{(j,\ell)\\ \ell-j\geq \tunif}} \frac{1}{n}\cdot p_k(x,w) p_{i-k}(w,v) p_j(y,v) + \sum_{v,w}\sum_{\substack{(i,k)\\i\geq k}}\sum_{\substack{(j,\ell)\\ \ell-j\leq  \tunif}}p_k(x,w) p_{i-k}(w,v) p_j(y,v) p_{\ell-j}(v,w)\\
& \lesssim \frac{1}{n} \sum_v \sum_{\substack{(i,k)\\i\geq k}} \sum_{\substack{(j,\ell)\\ \ell-j\geq  \tunif}}p_i(x,v) p_j(y,v) +
\sum_{v,w}\sum_{k}\sum_{\substack{(j,\ell)\\ \ell-j\leq  \tunif}}\left(\sum_{i<k+\tunif}\frac{1}{\sqrt{i-k}} + \frac{t}{n}\right)
p_k(x,w) p_j(y,v) p_{\ell-j}(v,w)
\\
&\lesssim\frac{1}{n} \sum_{\substack{(i,k)\\ i\geq k}} \sum_{\substack{(j,\ell)\\ \ell-j\geq  \tunif}}p_{i+j}(x,y)+ \sqrt{\tunif}\cdot \sum_w \sum_k\sum_{\substack{(j,\ell)\\ \ell-j\leq  \tunif}} p_\ell(y,w) p_k(x,w) \\&
= \frac{1}{n} \sum_{\substack{(i,k)\\i\geq k}} \sum_{\substack{(j,\ell)\\ \ell-j\geq \tunif}}p_{i+j}(x,y)+
\sqrt{\tunif} \cdot \sum_k \sum_{\substack{(j,\ell)\\ \ell-j\leq  \tunif}} p_{k+\ell}(x,y)
\\
& \leq \frac{t^2}{n} \sum_{i=0}^{t}\sum_{j=0}^{t} p_{i+j}(x,y) + \sqrt{\tunif}\cdot \sum_j \sum_{L\leq \tunif} \sum_k p_{k+L+j}(x,y)\\
& \leq \frac{t^2}{n} \cdot A + \sqrt{\tunif} \cdot (\tunif) \cdot \left(\frac{t^2}{n} + (\tunif)^{\frac{3}{2}} \right)
\asymp \frac{t^2}{n}\cdot A + (\tunif)^{3}\asymp (\tunif)^{\frac{3}{2}}\cdot A + (\tunif)^{3},
\end{align*}
where the third inequality follows since $\tunif\lesssim n^2$ for all regular graphs and the last inequality follows from Lemmas~\ref{lem:returnpro} and~\ref{lem:cauchy}, i.e.\ for $t\lesssim n^2$
\[
p_t(x,y) \leq \sqrt{p_t(x,x)}\sqrt{p_t(y,y)}\lesssim \frac{1}{\sqrt{t}}.
\]
Therefore, applying the second moment method we now get
\begin{align*}
\prstart{I_t>0}{x,y} \geq \frac{\left(\estart{I_t}{x,y}\right)^2}{\estart{I_t^2}{x,y}} \gtrsim \frac{A^2}{\left(\tunif\right)^{\frac{3}{2}} A + (\tunif)^3} = \frac{A}{\left(\tunif \right)^{\frac{3}{2}} + \frac{(\tunif)^3}{A}}.
\end{align*}
Since $\tunif\lesssim n^2$ for any regular graph, we can take $c$ large enough to ensure that $t-\tunif\asymp t$. Thus we get that the quantity $A$ can be lower bounded by
\begin{align*}
A \geq \sum_{\substack{i,j\leq t\\i+j>\tunif}} p_{i+j}(x,y) \gtrsim \frac{t^2}{n} \asymp \left( \tunif\right)^{\frac{3}{2}}.
\end{align*}
Since the function $f(x) = x/(1+1/x)$ is increasing for $x>0$, using the above lower bound on~$A$,
we finally conclude that
\begin{align*}
\prstart{I_t>0}{x,y} \geq c'>0.
\end{align*}
Since the above bound holds uniformly for all $x$ and $y$ we can
perform independent experiments to finally conclude that for regular graphs $t_I \lesssim \sqrt{n} \left(\tunif \right)^{\frac{3}{4}}$.
\end{proof}

\bibliographystyle{plain}
\bibliography{biblio}

\end{document}